\documentclass{article}

\usepackage{hyperref}
\usepackage{amsmath}
\usepackage{amsthm}
\usepackage{amssymb}
\usepackage{cite}
\usepackage{amsmath,amsthm}
\usepackage{amsfonts}
\usepackage{amssymb}
\usepackage{mathtools}
\usepackage{tikz}
\usepackage{enumitem}
\usepackage{lipsum}
\newcommand\blfootnote[1]{%
  \begingroup
  \renewcommand\thefootnote{}\footnote{#1}%
  \addtocounter{footnote}{-1}%
  \endgroup
}

\usepackage{dirtytalk} 
\usepackage{import}
\newtheorem{theorem}{Theorem}[section]

\newtheorem{conjecture}[theorem]{Conjecture}
\newtheorem{corollary}[theorem]{Corollary}
\newtheorem{proposition}[theorem]{Proposition}
\newtheorem{lemma}[theorem]{Lemma}

\newtheorem{claim}[theorem]{Claim}

\usepackage{lineno}

\begin{document}

\title{\bf A Cantor-Bendixson Rank for Siblings of Trees 
\blfootnote{This paper is a project as part of author's thesis under supervision of Dr. Claude Laflamme and Dr. Robert Woodrow at the Department of Mathematics and Statistics, University of Calgary, Calgary, AB, Canada (2017-2022).}}

\author{Davoud Abdi} 

\maketitle              

\begin{abstract}
Similar to topological spaces, we introduce the Cantor-Bendixson rank of a tree $T$ by repeatedly removing the leaves and the isolated vertices of $T$ using transfinite recursion. Then, we give a representation of a tree $T$ as a leafless tree $T^\infty$ with some leafy trees attached to $T^\infty$.  With this representation at our disposal, we count the siblings of a tree and obtain partial results towards a conjecture of Bonato and Tardif. 
\end{abstract}

\section{Introduction}

Trees in this literature are in the graph theoretical sense, that is, connected and acyclic simple graphs.
An {\em embedding} from a tree $T$ to another tree $S$ is an injective map from the vertex set of $T$ to the vertex set of $S$ preserving the adjacency relation. An embedding of a tree $T$ is an embedding from $T$ to itself. The set of all embeddings of a tree $T$ forms a monoid under composition of functions called the {\em monoid of embeddings} of $T$, denoted by $Emb(T)$. Two trees are called {\em equimorphic} or {\em siblings} if there are mutual embeddings between them. Clearly, two equimorphic finite trees are isomorphic. But, this is no longer the case for infinite trees. For instance, a tree consisting of a vertex $r$ and countably many paths of length 2 attached to $r$ has countably many siblings, up to isomorphism. The number of isomorphism classes of siblings of a tree $T$ is called the {\em sibling number} of $T$, denoted by $Sib(T)$. 
Bonato and Tardif \cite{BT} made the following conjecture. 

\begin{conjecture}[The Tree Alternative Conjecture, \cite{BT}]
If T is a tree, then $Sib(T)=1$ or $\infty$.
\end{conjecture}

Bonato and Tardif \cite{BT} proved their conjecture for rayless trees and the conjecture was verified for rooted trees by Tyomkyn \cite{TY}.
Laflamme, Pouzet, Sauer \cite{LPS} used Halin's fixed point theorem to prove the tree alternative  conjecture for scattered trees (trees not containing a subdivision of the complete binary tree). Indeed, they verified the conjecture for a more general class of trees namely stable trees. Hamann \cite{HAM} used a result by Laflamme, Pouzet and Sauer to prove that a tree either is non-scattered or has a vertex, an edge, an end or two ends fixed by all its embeddings. He also made use of the monoid of embeddings and deduced that the tree alternative conjecture holds for trees not satisfying two specific structural properties of that monoid. Later Abdi \cite{A} showed that a tree satisfying one of those properties is stable, and therefore the tree alternative conjecture also holds in that case. Tateno \cite{TAT} claimed a counterexample to the Bonato-Tardif conjecture in his thesis. Abdi, Laflamme, Tateno and Woodrow \cite{ALTW} revisited and verified Tateno's claim and provided  locally finite trees having an arbitrary finite number of siblings. 
This is a major development in the programme of understanding siblings of a given tree. While counting the number of siblings provides a good first insight into the siblings of a tree, in particular understanding which trees exactly do satisfy the dichotomy, the equimorphy programme is now ready to move on and focus on the actual structure of those siblings.

In topology, the {\em derived set} of a subset  of a topological space is the result of removing all isolated points from its closure. The concept of derived set was first introduced by Cantor in 1872 and he developed set theory in large part to study derived sets on the real line (see \cite{S}). 
The Cantor-Bendixson rank of a topological space $X$ is obtained by repeatedly defining derived sets  using transfinite recursion (see \cite{W}). For a tree $T$, we introduce a similar notion, namely {\em the Cantor-Bendixson rank} of $T$, by repeatedly removing the leaves and the isolated vertices of $T$ using transfinite recursion. Then, we show that $T$ can be represented as a leafless tree $T^\infty$ with some leafy trees attached to $T^\infty$. This representation of a tree $T$ and also the rank of $T$ will help us to count the siblings of $T$ and obtain partial results towards the tree alternative conjecture.

\section{Leaf Representation of Trees}

In this section we represent a tree $T$ of an arbitrary cardinality as a leafless tree to which some trees with leaves are attached. This representation helps us to determine the sibling number of $T$ in some cases. 
Let $T$ be a tree. The number of neighbours of a vertex $v\in T$ is called the {\em degree} of $v$, denoted by $deg(v)$. A vertex in $T$ is called a \textit{leaf} if its degree is 1 (\cite{D}). If $T'$ is a subtree of $T$ and $v\in T'$, by $deg_{T'}(v)$ we mean the degree of $v$ in $T'$. A tree is called {\em leafless} when it does not have leaf. We define a sequence $T^\alpha$ of subtrees of $T$ ($\alpha$ is an ordinal) by transfinite recursion as follows.

Let $T$ be a tree. Set
\begin{enumerate}
    \item   $T^0=T$;
    \item  If $\alpha$ is an ordinal and $T^\alpha$ is defined, then $T^{\alpha+1}=T^\alpha\setminus \{v\in T^\alpha : deg_{T^\alpha}(v)\leq 1\}$;
    \item  If $\lambda$ is a limit ordinal and $T^\alpha$ is defined for every $\alpha<\lambda$, then $T^\lambda=\bigcap_{\alpha < \lambda} T^\alpha$.
\end{enumerate}

We have the following observation.  

\begin{proposition}
Let $T$ be a tree. For every ordinal $\alpha$, $T^\alpha$ is empty or it is a subtree of $T$.
\end{proposition}

\begin{proof}
We prove it by transfinite induction. Let $\mathcal{S}(\alpha)$ be the statement: \say{$T^\alpha$ is empty or it is a tree}. Clearly $\mathcal{S}(0)$ holds. Assume that $\mathcal{S}(\alpha)$ holds. Since $T^\alpha$ is acyclic, so is $T^{\alpha+1}$. It remains to prove that $T^{\alpha+1}$ is connected. Pick two vertices $u,v$ in $T^{\alpha+1}$. Since $u,v$ are also vertices of $T^\alpha$, they are connected by a path $P$ in $T^\alpha$, that is $P\subseteq T^\alpha$. Note that $deg_{T^\alpha}(x)\geq 2$ for every $x\in P$. Thus, $P\subseteq T^{\alpha+1}$. 
Finally, suppose $\lambda$ is a limit ordinal and $\mathcal{S}(\alpha)$ holds for every $\alpha<\lambda$. If $|T^\lambda|\leq 1$, then $\mathcal{S}(\lambda)$ holds. Otherwise pick two arbitrary vertices $u,v$ in $\bigcap_{\alpha<\lambda} T^\alpha$. Since these vertices belong to every $T^\alpha$, they are connected by a path $P$ in $T^\alpha$. Since $T^\alpha$ is a tree, this path is unique. Thus, the path $P$ is the same for all $\alpha <\lambda$. The path $P$ witnesses that $u, v$ are connected in $T^\lambda$. Further, since every $T^\alpha$ is acyclic, so is $T^\lambda$. 
\end{proof}

Let $T$ be a tree. The \textit{Cantor-Bendixson rank}, or the {\em rank}, of $T$, denoted by $rank(T)$, is the least ordinal $\alpha$ such that $T^{\alpha+1}=T^\alpha$.
If $\alpha$ is the rank of $T$, then we denote $T^\alpha$ by $T^\infty$. Note that $T^\infty$ might be empty (for instance, when $T$ is finite). Suppose $T^\infty\neq\emptyset$ and let $v\in T^\infty$. Then $v\in T^\alpha = T^{\alpha+1}$. Therefore, if $deg_{T^\alpha}(v)=1$, then $v\notin T^{\alpha+1}$, a contradiction. This means that $T^\infty$ is leafless. When $T^\infty\neq\emptyset$, we call a maximal non-trivial subtree of $T$ which is edge-disjoint from $T^\infty$ a {\em leafy branch} of $T$ and denote it by $H_r$ where $r$ is the unique vertex common between $H_r$ and $T^\infty$.  When  $T^\infty$ is empty, $T$ itself is the maximal subtree of $T$ which is edge-disjoint from $T^\infty$. If $T^\infty\neq\emptyset$, then 
we represent $T$ as $T:= \bigoplus_{i\in I}H_{r_i}\oplus^E  T^\infty$ where we use the notation $\oplus^E$ to indicate that the leafy branches $H_{r_i}$ are edge-disjoint from $T^\infty$. Note that the leafy branches $H_{r_i}$ are pairwise disjoint. We call this representation the {\em leaf representation} of $T$. The leaf representation of a tree $T$ is an edge-disjoint decomposition of $T$
into the leafless tree $T^\infty$ and the leafy branches of $T$. 

   
A {\em ray}, resp {\em double ray}, is a one-way, resp two-way, infinite path (\cite{D}). The next proposition implies that the leaf representation is not applicable for rayless trees and for trees with only one end.

\begin{proposition} \label{Uniondoublerays}
Let $\mathfrak{D}$ be the set of all double rays  in $T$. Then, $T^\infty=\bigcup_{Z\in\mathfrak{D}}Z$.
\end{proposition}

\begin{proof}
Let $\alpha$ be the rank of $T$ i.e. $T^\alpha=T^\infty$. If $Z$ is a double ray in $T$, we argue by transfinite induction that $Z\subseteq T^\beta$ for all $\beta<\alpha$. First we have $Z\subseteq T^0$ by assumption. If $Z\subseteq T^\beta$, then since $Z$ has no leaf, $Z\subseteq T^{\beta+1}$ and if $\alpha$ is a limit ordinal and $Z\subseteq T^\beta$ for every $\beta < \alpha$, then $Z\subseteq \bigcap_{\beta < \alpha}T^\beta$. In particular, $Z\subseteq T^\alpha$. The argument is true for every double ray $Z$ in $T$. Therefore, we have $\bigcup_{Z\in\mathfrak{D}}Z\subseteq T^\infty$. 

Take some $x_0\in T^\infty$.  Since $T^\infty$ is leafless, it follows that $deg_{T^\infty}(x_0)\geq 2$. Let $x_{-1}$ and $x_1$ be two neighbours of $x_0$ in $T^\infty$. For each $k\geq 1$, let $x_{-k}, \ldots, x_k\in T^\infty$ be selected such that $P^k=x_{-k}\cdots x_0\cdots x_k$ is a path. Since $T^\infty$ is leafless and has no cycle, $x_{-k}$, resp $x_k$, has at least one neighbour $x_{-k-1}$, resp $x_{k+1}$, in $T^\infty$ other than the vertices of $P^k$. Then $P^{k+1}=x_{-k-1}\cdots x_0 \cdots x_{k+1}$ is a path in $T^\infty$. Set $Z:=\bigcup_{k<\omega} P^k$ which is a double ray containing $x_0$. The double ray $Z$ is a witness to $x_0\in \bigcup_{Z\in\mathfrak{D}}Z$, that is, $T^\infty \subseteq \bigcup_{Z\in\mathfrak{D}}Z$. This completes the proof. 
\end{proof}

Let $T$ be a tree. Two rays $R_1, R_2$ in $T$ are called {\em equivalent}, denoted by $R_1\sim R_2$, if their intersection is also a ray. The equivalence classes of $\sim$ are called the {\it ends} of $T$. The set of ends of $T$ is denoted by $\Omega(T)$ (\cite{HAL}). 

\begin{corollary} \label{Leaflessempty}
Let T be a tree. If $|\Omega(T)|\leq 1$, then $T^\infty=\emptyset$. 
\end{corollary}

\begin{corollary} \label{Oneendinfrank}
Let T be a tree. If $|\Omega(T)|=1$, then $rank(T)=\infty$. 
\end{corollary}

\begin{proof}
By Corollary \ref{Leaflessempty}, $T^\infty=\emptyset$. Since $|\Omega(T)|=1$, the existence of a ray in $T$ ensures that there are infinitely many steps to remove all vertices of the ray meaning that $rank(T)=\infty$. 
\end{proof}

Let $T$ be a tree. 
Corollary \ref{Leaflessempty} implies that if $T^\infty\neq \emptyset$, then $T$ has more than one end and consequently there is a double ray $Z$ in $T$. The double ray lies in $T^\infty$ by Proposition \ref{Uniondoublerays}. Assume that $T^\infty\neq \emptyset$. If some leafy branch $H_r$ of $T$ contains a ray, then it contains a ray $R$ whose starting vertex is $r$. Let $R'$ be a ray in $T^\infty$ whose starting vertex is $r$. Then, the tree consisting of $R\cup R'$ is a double ray with infinitely many vertices in $H_r$, a contradiction because the leafy branches of $T$ are edge-disjoint from $T^\infty$. Thus, for each leafy branch $H_r$ of $T$ we have $|\Omega(H_r)|=0$ and by Corollary \ref{Leaflessempty},  $H_r^\infty=\emptyset$.

The following lemma provides the connection between the rank of $T$ and the ranks of its leafy branches. 

\begin{lemma} \label{Boundonrank} 
Let T be a tree. T is of finite rank if and only if there is a finite bound on the ranks of its leafy branches. 
\end{lemma} 

\begin{proof}
($\Rightarrow$) Assume that $rank(T)=n<\infty$. It follows that $T^m=T^n$ for each $m\geq n$. If  there is no finite bound on the $rank(H_r)$ where $r\in T^\infty$, then there is a leafy branch $H_s$ of $T$ with $rank(H_s) > n$. This implies that there is a leaf in $T^n$ meaning that $T^{n+1}\neq T^n$, a contradiction. 

($\Leftarrow$) Suppose that $M<\infty$ is a bound on the $rank(H_r)$ where $r\in T^\infty$. Then, the maximum distance of a leaf of $T$ from $T^\infty$ is $M$. This means that $T^{n+1}=T^n$ for each $n\geq M$. Consequently, $rank(T)\leq M$. 
\end{proof}

\section{Siblings of Trees by means of Leafy Branches}

In this section we use the leaf representation of a tree $T$ to count its siblings. A \textit{rooted} tree $(T,r)$ is a tree $T$ with a special vertex $r$, called the root. Two rooted trees $(T,r)$ and $(T',r')$ are siblings if there are embeddings $f:(T,r)\to (T',r')$ and $g:(T',r')\to (T,r)$ such that $f(r)=r'$ and $g(r')=r$. It follows that all embeddings of a rooted tree $(T,r)$ fix the root (\cite{TY}).

\begin{lemma} \label{Leaftoleafless}
Let T and S be trees and $f:T\to S$ an embedding.
\begin{enumerate}
    \item $f(T^\infty)\subseteq S^\infty$. Thus, if $T\approx S$, then $T^\infty\approx S^\infty$. Moreover, if for some leafy branch $H_r$ of T, $f(H_r\setminus\{r\})\cap S^\infty\neq\emptyset$, then there is a ray in $S\setminus f(T)$. Consequently, if $f(T^\infty)\subset S^\infty$, then there is a ray in $S\setminus f(T)$.  
    \item If f is an isomorphism, then $f(T^\infty)=S^\infty$ and for each leafy branch $H_t$ of T, $f((H_t,t))=(H_s,s)$ for some leafy branch $H_s$ of S. 
\end{enumerate}
\end{lemma}

\begin{proof} 
(1) If $T^\infty=\emptyset$, then clearly we have $f(T^\infty)\subseteq S^\infty$. Assume that $T^\infty\neq\emptyset$ and that for some $x\in T^\infty$ we have $f(x)\in H_s\setminus\{s\}$ where $H_s$ is a leafy branch of $S$. Since  $x\in T^\infty$, by Proposition \ref{Uniondoublerays}, $T$ has a double ray $Z$ containing $x$. Then, $f(Z)$ is a double ray in $S^\infty$ containing $f(x)$. Since $f(x)\in H_s\setminus \{s\}$, it follows that $S^\infty$ and $H_s$ have at least one common edge which is not possible. Therefore, $f(T^\infty)\subseteq S^\infty$. 

Assume that for some $x\in H_r\setminus \{r\}$ where $H_r$ is a leafy branch of $T$, $f(x)\in S^\infty$. Since $x\neq r$, $x$ is not a vertex of a double ray in $T$. Therefore, $x$ has only one branch containing a ray while $f(x)$ has at least two branches each of which contains a ray because $f(x)\in S^\infty$ and by Proposition \ref{Uniondoublerays}, $S^\infty$ consists of all double rays of $S$. Hence, there is a ray in $S\setminus f(T)$. 
Now assume that $f(T^\infty)\subset S^\infty$. Pick some $t_1\in S^\infty\setminus f(T^\infty)$. For each $n\geq 1$, since $S^\infty$ is leafless and acyclic, $t_n$ has a neighbour $t_{n+1}\in S^\infty\setminus f(T^\infty)$ such that $t_{n+1}\neq t_m$ for each $m\leq n$. Set $R:=t_1t_2\cdots$ which is a ray in $S^\infty\setminus f(T^\infty)$. If no vertex of $R$ is in the image of $f$, then $R\subset S\setminus f(T)$. If for some $n$, $t_n=f(x)$ where $x\in T\setminus T^\infty$, then by the argument above, there is a ray  (indeed, a tail of $R)$ in $S\setminus f(T)$. 

(2) Suppose that $f:T\to S$ is an isomorphism. By (1), it follows that $f(T^\infty)=S^\infty$. Now let $H_t$ be a leafy branch of $T$. If $f(t)\in S^\infty\setminus \bigcup_{s\in S^\infty} H_s$, then some neighbour of $t$ in $H_t$ is mapped by $f$ to some element of $S^\infty$. By (1), there is a ray in $S^\infty$ with no preimage under $f$, a contradiction because $f$ is an isomorphism. Therefore, $f(H_t,t)=(H_s,s)$ for some leafy branch $H_s$ of $S$.  
\end{proof} 

As we will see later, when $T\approx S$, then it is not necessarily the case that $T^\infty \cong S^\infty$.

\begin{theorem} \label{Leafless}
If $T$ is a leafless tree, then $Sib(T)=1$ or $\infty$. 
\end{theorem}   

\begin{proof} 
Note that $T^\infty=T$. Suppose that there is a non-surjective embedding $f$ of $T$. Then $S:=f(T)$ is a proper and leafless subtree of $T$. By Lemma \ref{Leaftoleafless} (1) there is a ray $R=t_1t_2\ldots$ in $T\setminus S$. Without loss of generality assume that $t_1$ is a neighbour of some $y\in S$. For each $n\geq 1$, let $T_n$ be the subtree of $T$ with vertex set $S\cup P_n$ where $P_n:=t_1\cdots t_n$. We have $T_n^\infty=S$ for each $n$. Therefore, $T_n=S\oplus^E H_y^n$ where $H_y^n$ consists of $y$ and $P_n$. Moreover, $T\hookrightarrow S\hookrightarrow T_n\hookrightarrow T$ meaning that $T_n \approx T$ for each $n\geq 1$. Since $T$ is leafless, $T\ncong T_n$ for every $n$. Now suppose that for $m< n$, $f: T_n \to T_m$ is an isomorphism. Then by Lemma \ref{Leaftoleafless} (2), $f(T^\infty_n)=T^\infty_m$ and $f((H_y^n,y))=(H_y^m,y)$ because $H_y^n$ and $H_y^m$ are the only leafy branches of $T_n$ and $T_m$, respectively. But this is not possible since $H_y^n$ and $H_y^m$ are paths of length $n$ and $m$, respectively. It follows that $T_n \ncong T_m$ when $n\neq m$. Hence, $Sib(T)=\infty$. 

In particular, if $Sib(T)>1$, then there is a non-isomorphic sibling $S\subset T$ of $T$ and consequently there exists a non-surjective embedding $f$ of $T$ with $f(T)\subseteq S\subset T$. The argument above implies that $Sib(T)=\infty$. 
\end{proof}

The {\em complete binary tree} is the tree in which one vertex is of degree 2 and all others have degree 3. Let $T$ be the complete binary tree and let $r$ be the unique vertex of $T$ with degree 2. We note that $T$ is leafless, that is $T^\infty=T$. 
Let $r_1$ be one of the two neighbours of $r$ and $R=rr_1r_2\cdots$ a ray in $T$ starting at $r$ and containing $r_1$. Now let $S$ be the tree obtained from $T$ by replacing the maximal subtree of $T$ containing $r_1$ and edge-disjoint from $R$ with a trivial tree. We have $deg_S(r_1)=2$. It can be easily shown that $T\approx S$.  Moreover, $S$ has no leaf meaning that $S^\infty=S$. Further, there is only one vertex of degree 2 in $T$ while in $S$ there are precisely two such vertices, establishing that $T^\infty \ncong S^\infty$. Consequently, $T\ncong S$ and by Lemma \ref{Leafless} we have $Sib(T)=\infty$.

A leafy branch $H_r$ of a tree $T$ might have infinite sibling number as a rooted tree $(H_r,r)$. As a matter of fact, consider the rooted tree $(T,r)$ consisting of a vertex $r$ and countably many paths of length 2 attached to $r$ which has infinite sibling number. The next lemma shows that the existence of such a leafy branch of $T$ is a sufficient condition to conclude that the sibling number of $T$ is infinite. 

\begin{lemma} \label{Rootedcomponentwithinfinite}
Let $T$ be a tree. If  some leafy branch $H_r$ of $T$ has infinitely many siblings as a tree rooted at r, then $Sib(T)=\infty$. 
\end{lemma}

\begin{proof}
Suppose that $H_r$ is a leafy branch of $T$ with $Sib((H_r,r))=\infty$ and let $\{(H_{r_n},r_n)\}_{n< \omega}$  be a family of rooted trees which are pairwise non-isomorphic siblings of $(H_r,r)$. For every $n<\omega$, let $T_n$ be the tree obtained from $T$ by replacing each $(H_t,t)\approx (H_r,r)$ with a copy of $(H_{r_n},r_n)$ where $H_t$ is a leafy branch of $T$. Then the resulting trees $T_n$ are siblings of $T$. Now assume that for some $m<n<\omega$, $T_m\cong T_n$ by some isomorphism $f$. Let $H_t$ be a leafy branch of $T_m$ such that $(H_t,t)$ is equimorphic to $(H_r,r)$. Then $(H_t,t)=(H_{r_m},r_m)$. By Lemma \ref{Leaftoleafless} (2),  $f(T_m^\infty) = T_n^\infty$ and $f((H_{r_m},r_m))=(H_s,s)$ for some leafy branch $H_s$ of $T_n$. We have $(H_s,s)\cong (H_{r_m},r_m)\approx (H_r,r)$. Thus, $(H_s,s)=(H_{r_n},r_n)$ because $H_s$ is a leafy branch of $T_n$. It follows that $(H_{r_m},r_m)\cong (H_{r_n},r_n)$, a contradiction. Hence, the family $\{T_n\}_{n<\omega}$ is a witness to $Sib(T)=\infty$.  
\end{proof} 

Define a {\em comb} to be a ray with infinitely many disjoint non-trivial paths of finite length attached to it \cite{TY}.
Tyomkyn \cite{TY} proved that if a locally finite tree $T$ has an embedding $f$ such that $T\setminus f(T)$ contains a comb, then $T$ has infinitely many siblings. For an arbitrary tree $T$,
we get a similar result by posing some restrictions. 

\begin{lemma} \label{Infiniteleaf}
Let $T$ be a tree of finite rank. If for some embedding $f$ of $T$, $T\setminus f(T)$ contains a ray, then $Sib(T)=\infty$. 
\end{lemma} 

\begin{proof}
Assume that $f$ is an embedding of $T$ such that $T\setminus f(T)$ contains a ray $R=r_1r_2\cdots$ where $r_1$ has a neighbour $y$ in $f(T)$. Let $x\in T$ be the preimage of $y$ under $f$, that is $f(x)=y$.  For every $n$, let $T_n$ be the tree obtained from $T$ by attaching a path $P_n=t_1\cdots t_n$ of length $n-1$ to $x$ using an edge. The embedding $f$ can be extended to an embedding $f_n:T_n\to T$ by sending the vertices of the finite path $P_n$ to an initial segment of the ray $R$. Thus, we have $T\approx T_n$ for each $n$. For each $n$, let $H_s^n$ be the leafy branch of $T_n$ containing $x$ ($s$ might be equal to $x$). Note that for each $n$, $T^\infty_n=T^\infty$. By Lemma \ref{Boundonrank} there is a finite bound $M$ on the ranks of the $H_r$ where $H_r$ is a leafy branch of $T$. Let $n>m>M$ and $f:T_n\to T_m$ be an isomorphism.  By Lemma \ref{Leaftoleafless} (2), $f(T_n^\infty)=T_m^\infty$ and $f((H_s^n,s))=(H_r,r)$ for some leafy branch $H_r$ of $T_m$. But this is not possible because $T_n$ has a leafy branch $H_s^n$ containing the leaf $t_n$ at distance $n+k$ from $T^\infty$ for some non-negative integer $k$, while the maximum distance of a leaf of $T_m$ from $T^\infty$ is $m+k$. Hence,  for $n> m > M$, $T_n\ncong T_m$ meaning that $Sib(T)=\infty$.  
\end{proof}

Lemma \ref{Infiniteleaf} provides a useful tool that we will use in the proof of the following proposition.  Note that when a tree $T$ has no end, then $T$ is rayless and recall that Bonato and Tardif proved that a rayless tree has one or infinitely many siblings (see \cite{BT} Theorem 1). 

\begin{proposition} \label{Finiteleaf}
Let $T$ be a tree of finite rank with only finitely many leafy branches. Then $Sib(T)=1$ or $\infty$. 
\end{proposition} 

\begin{proof} 
First note that by Corollary \ref{Oneendinfrank} we have $|\Omega(T)|\neq 1$. If $T$ is rayless, then the statement holds by \cite{BT} Theorem 1. Assume that $T$ has more than one end which implies that $T^\infty\neq\emptyset$. 
If $T$ is leafless or some leafy branch $H_r$ of $T$ has infinite sibling number as a tree rooted at $r$, then the statement holds by Theorem \ref{Leafless} and Lemma \ref{Rootedcomponentwithinfinite}. Suppose $T$ has only finitely many  leafy branches $H_{r_1}, \ldots , H_{r_k}$, $k\geq 1$, such that each $(H_{r_i},r_i)$ has only one sibling. 

\noindent{\bf Case 1} 
There exists an embedding $f$ of $T$ such that $T\setminus f(T)$ contains a ray. Then, by Lemma \ref{Infiniteleaf}, $Sib(T)=\infty$.  

\noindent{\bf Case 2}
For no embedding $f$ of $T$, $T\setminus f(T)$  contains a ray. Let $f$ be an embedding of $T$. By Lemma \ref{Leaftoleafless} (1) we have $f(T^\infty)\subseteq T^\infty$. Also, if $f(T^\infty)\subset T^\infty$, then Lemma \ref{Leaftoleafless} (1) implies that there is a ray in $T\setminus f(T)$ contradicting our assumption. Therefore, $f(T^\infty)=T^\infty$. Moreover, for each $i$, $f(H_{r_i}\setminus\{r_i\})\cap T^\infty=\emptyset$ because otherwise by Lemma \ref{Leaftoleafless} (1) there is a ray in $T\setminus f(T)$. 

\begin{claim}
Each $r_i$ is mapped to some $r_j$ by $f$.
\end{claim}

\begin{proof}
If  $f(r_i)=y\in T^\infty\setminus\{r_1,\ldots, r_k\}$ for some $i$, then some neighbour $x\in H_{r_i}$ of $r_i$ is mapped to a neighbour $z$ of $y$. Since $y\in T^\infty\setminus\{r_1,\ldots, r_k\}$, it follow that $z\in T^\infty$, a contradiction. 
\end{proof}
\noindent 
Therefore, for each $1\leq i\leq k$, there is a unique $1\leq j\leq k$ such that  $f:(H_{r_i},r_i)\hookrightarrow (H_{r_j},r_j)$. Since $\{r_1,\ldots,r_k\}$ is a finite set, each $(H_{r_i},r_i)$ is equimorphic to the $(H_{r_j},r_j)$ with $j\in f.r_i$ where $f.r_i$ is the orbit of $r_i$ under $f$. By assumption, for each $i$, $Sib((H_{r_i},r_i))=1$ which implies that $(H_{r_i},r_i)\cong (H_{r_j},r_j)$ by $f$ for every $j\in f.r_i$. From $f(T^\infty)=T^\infty$ and the above fact, we conclude that $f$ is an automorphism of $T$. Thus, $Sib(T)=1$ in this case. 
\end{proof}

By Lemmas \ref{Rootedcomponentwithinfinite} and \ref{Infiniteleaf} and Proposition \ref{Finiteleaf} we get the following. 

\begin{theorem} \label{DichT}
Let T be a tree. If one of the following holds,  then  $Sib(T)=1$ or $\infty$. 
 
 \begin{enumerate}
     \item Some leafy branch $H_r$ of T  has infinitely many siblings as a tree rooted at r.
     \item T is of finite rank and for some embedding f of T, $T\setminus f(T)$ contains a ray.
     \item T is of finite rank with only finitely many leafy branches.
 \end{enumerate} 
\end{theorem}

\begin{proof}
(1) It follows by Lemma \ref{Rootedcomponentwithinfinite}. 

(2) If $T$ is of finite rank, then $|\Omega(T)|\neq 1$ by Corollary \ref{Oneendinfrank}. If $T$ has no end, then for no embedding $f$ of $T$, $T\setminus f(T)$ contains a ray. Therefore, $T$ has more than one end and by Lemma  \ref{Infiniteleaf} we have $Sib(T)=\infty$. 

(3) If $T$ is of finite rank, then $|\Omega(T)|\neq 1$ by Corollary \ref{Oneendinfrank}. If $T$ has no end, then $T$ is rayless and $Sib(T)=1$ or $\infty$ by \cite{BT} Theorem 1. If $T$ has more than one end, then $Sib(T)=1$ or $\infty$ by Proposition \ref{Finiteleaf}. 
\end{proof}

\vspace{0.5cm}
\noindent{\bf \large Acknowledgements}

I would like to thank my PhD supervisors Professor Robert Woodrow and Professor Claude Laflamme for suggesting this problem and for their help and advice. 


\vspace{1cm}

\textsc{Department of Mathematics and Statistics, University of Calgary, Calgary, Alberta, Canada, T2N 1N4}

{\em Email address:} 
\url{davoud.abdikalow@ucalgary.ca}


\begin{thebibliography}{5}

\bibitem{A}
D. Abdi, {\em On infinitely many siblings for locally finite trees with parabolic embeddings}, (2022), \url{https://arxiv.org/abs/2209.03897}

\bibitem{ALTW}
D. Abdi, C. Laflamme, A. Tateno, R. Woodrow, {\em An example of Tateno disproving conjectures of Bonato-Tardif, Thomasse, and Tyomkyn}, (2022), \url{https://arxiv.org/abs/2205.14679}.

\bibitem{BT}
A. Bonato, C. Tardif, \textit{Mutually embeddable graphs and the tree alternative conjecture}, J. Combin. Theory Ser. B 96 (2006), 874-880. 

\bibitem{D}
R. Diestel. \textit{Graph Theory}. Fifth ed. Springer. (2017).


\bibitem{HAL}
R. Halin. \textit{Automorphisms and endomorphisms of infinite locally finite graphs}. Abh. Math. Sem. Univ. Hamburg 39 (1973), 251-283. 


\bibitem{HAM}
M. Hamann. \textit{Self-embeddings of trees}. Discrete Math. 342 (2019), no 12. 111586, 1-7.


\bibitem{LPS}
C. Laflamme, M. Pouzet, N. Sauer. \textit{Invariant subsets of scattered trees and the tree alternative property of Bonato and Tardif}. Abh. Math. Semin. Univ. Hambg. 87 (2017), 369-408. 

\bibitem{S}
J. Stillwell, {\em The Real Numbers, An Introduction to Set Theory and Analysis}, Springer, (2013).

\bibitem{TAT}
A. Tateno, \textit{Mutually embeddable trees and a counterexample to the Tree Alternative Conjecture}, unpublished manuscript, 32 pages, (2008). 

\bibitem{TY}
M. Tyomkyn. \textit{A proof of the rooted tree alternative conjecture}. Discrete Math 309 (2009), 5963-5967. 

\bibitem{W}
C. Wurm, (D-DSLD-NDM)
{\em The Cantor-Bendixson analysis of finite trees}, (English summary) Formal grammar, 185–200,
Lecture Notes in Comput. Sci., 8612, Springer, Heidelberg, (2014).

\end{thebibliography}
\end{document}